\documentclass[reqno,a5paper]{amsart}
\usepackage{amsmath,amssymb,amsbsy,amsfonts,amsthm,latexsym,amsopn,amstext,amsxtra,euscript,amscd,amsthm}

\newtheorem{lemma}{Lemma}

\newtheorem{theorem}{Theorem}
\newtheorem*{theorem*}{Theorem}

\newcommand{\Z}{\mathbb Z}
\newcommand{\Q}{\mathbb Q}
\newcommand{\C}{\mathbb C}

\newcommand{\oc}{\rho}
\newcommand{\K}{K}

\newcommand{\acr}{\newline\indent}

\begin{document}

\title[Diophantine triples in recurrences of Pisot type]{Diophantine triples in linear recurrence sequences of Pisot type}

\author[C. Fuchs, C. Hutle\and F. Luca]{Clemens Fuchs*$^\dag$, Christoph Hutle* \and Florian Luca$^\ddag$}

\address{\llap{*\,}University of Salzburg \acr Hellbrunner Str. 34/I \acr 5020 Salzburg \acr AUSTRIA}
\email{clemens.fuchs@sbg.ac.at, christoph.hutle@gmx.at}

\address{\llap{$^\ddag$\,}School of Mathematics\acr University of the Witwatersrand\acr Private Bag X3, Wits 2050, South Africa\acr Max Planck Institute for Mathematics\acr  Vivatsgasse 7, 53111 Bonn, Germany\acr
Department of Mathematics\acr Faculty of Sciences\acr University of Ostrava\acr 30 Dubna 22, 701 03 Ostrava 1, Czech Republic}
\email{florian.luca@wits.ac.za}

\thanks{C.F. and C.H. were supported by FWF (Austrian Science Fund) grant No. P24574 and by the Sparkling Science project EMMA grant No. SPA 05/172. F.L. was supported in part by grant CPRR160325161141 and an A-rated scientist award both from the NRF of South Africa and by grant No. 17-02804S of the Czech Granting Agency.}
\thanks{$^\dag$Corresponding author.}

\maketitle

\begin{abstract}
The study of Diophantine triples taking values in linear recurrence sequences is a variant of a problem going back to Diophantus of Alexandria which has been studied quite a lot in the past. The main questions are, as usual, about existence or finiteness of Diophantine triples in such sequences. Whilst the case of binary recurrence sequences is almost completely solved, not much was known about recurrence sequences of larger order, except for very specialized generalizations of the Fibonacci sequence. Now, we will prove that any linear recurrence sequence with the Pisot property contains only finitely many Diophantine triples, whenever the order is large and a few more not very restrictive conditions are met.
\end{abstract}

\section{Introduction}
\label{sec:1}

The problem of Diophantus of Alexandria about tuples of integers $\{a_1, a_2, a_3,$ $\dots,a_m\}$ such that the product of each distinct two of them plus 1 always results in an integer square has already quite a long history (see \cite{dujella-webpage}). One of the main questions was, how many such Diophantine $m$-tuples exist for a fixed $m \geq 3$.
Already Euler proved that there are infinitely many Diophantine quadruples, demonstrating it with the family
$$
\{a, b, a+b+2 \sqrt{ab+1}, 4(a+ \sqrt{ab+1})(b+ \sqrt{ab+1}) \sqrt{ab+1}\}
$$
for $a$ and $b$ such that $ab + 1$ is a perfect square. However no algorithm for generating all quadruples has been found.

Much later Arkin, Hoggatt and Strauss \cite{AHS} proved that every Diophantine triple can be extended to a Diophantine quadruple. More precisely, let $\{a, b, c\}$ be a Diophantine triple and
$$
ab + 1 = r^2, \quad ac + 1 = s^2, \quad bc + 1 = t^2,
$$
where $r,s,t$ are positive integers. Define
$$
d_+ := a + b + c + 2abc + 2rst.
$$
Then $\{a, b, c, d_+\}$ is a Diophantine quadruple.

Dujella proved in \cite{Du}, that there are no Diophantine sextuples and also that there are only finitely many Diophantine quintuples. This result is even effective, since an upper bound of the form $\log_{10}(\log_{10}(\max\{{a_i}\})) < 26$ was given on the members of such a quintuple. It is conjectured, that there are no quintuples at all and, even stronger, that if $\{a, b, c, d\}$ is a Diophantine quadruple and $d > \max\{a, b, c\}$, then $d = d_+$. The ``weaker" conjecture has recently been settled by He, Togb\'e and Ziegler (cf. \cite{htz}), whereas the stronger conjecture still remains open. Here, Fujita and Miyazaki \cite{FM} proved that any fixed Diophantine triple can be extended to a Diophantine quadruple in at most 11 ways by joining a fourth element exceeding the maximal element in the triple.

Now it is an interesting variation of the original problem of Diophantus to consider a linear recurrence sequence instead of the sequence of squares. So we ask for bounds $m$ on the size of tuples of integers $\{a_1, a_2, a_3, \dots, a_m\}$ with $a_i a_j + 1$ being members of a given linear recurrence for $1\leq i<j\leq m$.
Here, the first result was due to Fuchs, Luca and Szalay, who proved in \cite{FuLuSz} that for a binary linear recurrence sequence $(u_n)_{n \geq 0}$, there are only finitely many Diophantine triples, if certain conditions are met. The Fibonacci sequence and the Lucas sequence both satisfy these conditions and all Diophantine triples with values in these sequences were computed in \cite{LuSz} and \cite{LuSz1}. Further results in this direction can be found in \cite{aiz2}, \cite{ISz1} and \cite{LuMu}. Moreover, in \cite{aiz1} it is shown that there are no balancing Diophantine triples; see also \cite{aiz3} for a related result. In \cite{ai} it is shown that there are no Diophantine triples taking values in Pellans sequence.

The first result on linear recurrence sequences of higher order than 2 came up in 2015, when the authors jointly with Irmak and Szalay proved (see \cite{FuLuSz1}) that there are only finitely many Diophantine triples with values in the Tribonacci sequence $(T_n)_{n \geq 0}$ given by
$$
T_0 =T_1 =0, \quad T_2 =1, \quad T_{n+3} =T_{n+2}+T_{n+1}+T_n \quad \textrm{for } n \geq 0.
$$
In \cite{GoLu} it was shown that a Tribonacci Diophantine quadruple does not exist. A related result can be found in \cite{RL1}.
One year later in \cite{FuLuSz2}, this result was generalized to $k$-generalized Fibonacci sequences: For any integer $k \geq 3$, define $(F_n^{(k)})_{n \geq 0}$ by $F_0^{(k)}=\ldots=F_{k-2}^{(k)}=0,F_{k-1}^{(k)}=1$ and
$$
F_{n+k}^{(k)}=F_{n+k-1}^{(k)}+\cdots+F_n^{(k)} \quad \textrm{for } n \geq 0.
$$
Then for any fixed $k$, only finitely many Diophantine triples with values in $\{F_n^{(k)}; n\geq 0\}$ exist.
None of these results are constructive, since the proof uses a version of the Subspace theorem. It is not clear, whether there are any Diophantine triples with values in those sequences at all.


The result in this paper deals with a significantly larger class of linear recurrence sequences:

Let $(F_n)_{n\geq 0}$ be a sequence of integers satisfying a linear recurring relation. Assume that the recurrence is of \emph{Pisot type}; i.e., that its characteristic polynomial is the minimal polynomial (over $\mathbb{Q}$) of a Pisot number. We denote the power sum representation (Binet formula) by $F_n = f_1 \alpha_1^n + \cdots + f_k \alpha_k^n$. Assume w.l.o.g. that $\alpha=\alpha_1$ is the Pisot number; i.e., $\alpha$ is a real algebraic integer of degree $k$ with $\vert\alpha\vert>1$ and $\alpha_2,\ldots,\alpha_k$ are the conjugates of $\alpha$ over $\mathbb{Q}$ and satisfy $\max\{|\alpha_2|, \dots, |\alpha_k|\} < 1$. We remark that by a result of Mignotte (cf. \cite{MM}) it immediately follows that the sequence is non-degenerate, and that the characteristic roots are all simple and irrational.

Then we will show, that there are only finitely many triples of integers $1\le a<b<c$ such that
\begin{displaymath}
1+ab=F_x,\quad 1+ac=F_y,\quad 1+bc=F_z
\end{displaymath}
if at least one of the following conditions hold:
\begin{itemize}
\item Neither the leading coefficient $f_1$ nor $f_1\alpha$ is a square in $\K=\mathbb{Q}(\alpha_1, \dots, \alpha_k)$.
\item $k \geq 5$ and $\alpha$ is not a unit in the ring of integers of $\K$.
\item $k \geq 6$.
\end{itemize}

The previously treated $k$-generalized Fibonacci sequences satisfy this Pisot property and neither their leading coefficient $f_1$ nor $f_1\alpha_1$  is a square. However, the  new result in this paper helps us to obtain finiteness for many more linear recurrence sequences.

For example, let us consider the irreducible polynomial $X^3 - X - 1$, which has the Pisot property. Its Pisot root $\theta := 1.3247179572 \dots$ is the smallest existing Pisot number by \cite{MJB}. This number is also known as the \emph{plastic constant}.
Its corresponding linear recurrence sequence $(F_n)_{n \geq 0}$, given by $F_{n+3} = F_{n+1} + F_n$, is of Pisot type. If the initial values are not $F_0 = 6, F_1 = -9, F_2 = 2$, then neither the leading coefficient nor the leading coefficient times $\theta$ are squares in the splitting field of $X^3-X-1$ over $\mathbb{Q}$. So the theorem can be applied and we obtain, that there are only finitely many Diophantine triples with values in this sequence.
However it is yet not clear, what happens in the case $F_0 = 6, F_1 = -9, F_2 = 2$.

Another example for which the theorem can be applied is the polynomial
$$
X^{2k+1} - \frac{X^{2k}-1}{X - 1}.
$$
This polynomial defines a Pisot number of degree $2k+1$ by a result of Siegel (see \cite{CLS}) and its corresponding linear recurrence sequence is of Pisot type. Independently of its initial values, the result applies to all $k \geq 3$ since the degree is sufficiently large.
The same applies to
$$
X^{2k+1} - \frac{X^{2k+2}-1}{X^2 - 1},
$$
for $k \geq 3$.

Furthermore, all polynomials of the form
$$
X^k(X^2-X-1) + X^2 + 1
$$
are known to be Pisot numbers. So, again for $k \geq 3$ the theorem applies.


We quickly discuss the main shape of the recurrences we study in this paper. Let $(F_n)_{n\geq 0}$ be a recurrence of Pisot type as described above. Put $\K=\mathbb{Q}(\alpha_1,\ldots,\alpha_k)$. Since $F_n\in\Z$ it follows that each element of the Galois group of $K$ over $\Q$ permutes the summands in the power sum representation of $F_n$. Moreover, each summand is a conjugate of the leading term $f_1\alpha_1^n$ over $\Q$ and each conjugate of it appears exactly once in the Binet formula.
Therefore $F_n$ is just the trace $\operatorname{Tr}_{K/\mathbb{Q}}(f_1 \alpha_1^n)$.
Since $f_1$ might not be integral, we write $f_1 = f/d$ with $d \in \Z$ and $f$ being an integral element in $K$.
Thus, conversely starting with a Pisot number $\alpha$, an integer $d \in \Z$ and an integral element $f$ in the Galois closure $K$ of $\alpha$ over $\Q$ such that $dF_n=\mbox{Tr}_{K/\Q}(f\alpha^n)$ for every $n\in\mathbb{N}$, we can easily construct further examples for which our result applies.


The proof will be given in several steps: First, a more abstract theorem is going to be proved, which guarantees the existence of an algebraic equality, that needs to be satisfied, if there were infinitely many Diophantine triples. This works on utilizing the Subspace theorem and a parametrization strategy in a similar manner to that of \cite{FuLuSz2}. If the leading coefficient is not a square in $\mathbb{Q}(\alpha_1, \dots, \alpha_k)$, we obtain the contradiction quite immediately from this equality.
In a second step, we will use B\'ezout's theorem, suitable specializations and algebraic parity considerations in order to show that this equality can also not be satisfied if the order $k$ is large enough.
Let us now state the results.

\section{The results}

We start with a general and more abstract statement which gives necessary conditions in case infinitely many Diophantine triples exist. It is derived by using the Subspace theorem.

\begin{theorem}
\label{th:main1}
Let $(F_n)_{n\geq 0}$ be a sequence of integers satisfying a linear recurrence relation of Pisot type of order $k\geq 2$. Denote its power sum representation as
$$
F_n = f_1 \alpha_1^n + \cdots + f_k \alpha_k^n.
$$
If there are infinitely many positive integers $1 < a < b < c$, such that
\begin{equation}
\label{eq:1}
ab+1=F_x,\qquad ac+1=F_y,\qquad bc+1=F_z
\end{equation}
hold for integers $x,y,z$, then one can find fixed integers $(r_1,r_2,r_3,s_1,s_2,s_3)$ with $r_1,r_2,r_3$ positive, $\gcd(r_1,r_2,r_3)=1$ such that infinitely many of the solutions $(a,b,c,x,y,z)$ can be parametrized as
$$
x=r_1 \ell +s_1,\quad y=r_2 \ell +s_2,\quad z=r_3 \ell+s_3.
$$
Furthermore, following the parametrization of $x, y, z$ in $\ell$, there must exist a power sum $C_\ell$ of the form
$$
C_\ell = \alpha_1^{(-r_1+r_2+r_3)\ell + \eta} \left(e_0+\sum_{j\in J_C} e_j \prod_{i=1}^k \alpha_i^{v_{ij} \ell} \right)
$$
with $\eta \in \mathbb{Z} \cup (\mathbb{Z} + 1/2)$, $J_C$ an index set, $e_j$ being coefficients in $\mathbb{Q}(\alpha_1, \ldots, \alpha_k)$ and integers $v_{ij}$ with the property that $v_{ij} \geq 0$ if $i \in \{2, \dots, n\}$ and $v_{ij} < 0$ if $i = 1$,
such that
$$
(F_x - 1) C_\ell^2 = (F_y - 1) (F_z - 1).
$$
Similarly there are $A_\ell$ and $B_\ell$ of the same shape with
$$
(F_z - 1) A_\ell^2 = (F_x - 1) (F_y - 1) \quad \textrm{and} \quad (F_y - 1) B_\ell^2 = (F_x - 1) (F_z - 1)
$$
\end{theorem}

The proof is given in Sections \ref{sec:2}, \ref{sec:4} and \ref{sec:5}.

This theorem looks quite abstract. However, it can be applied to a huge family of linear recurrences.
Firstly, it can be applied to all linear recurrences, in which the leading coefficient is not a square:
\begin{theorem}\label{co:nonsquare}
Let $(F_n)_{n\geq 0}$ be a linear recurrence sequence of integers with
$$
F_n = f_1 \alpha_1^n + \cdots + f_k \alpha_k^n
$$
which satisfies the conditions of Theorem \ref{th:main1}. If furthermore neither $f_1$ nor $f_1 \alpha_1$ are squares in $\mathbb{Q}(\alpha_1, \dots, \alpha_k)$, then there are only finitely many Diophantine triples with values in $\{F_n; n\geq 0\}$.
\end{theorem}
The proof of this theorem is given in Section \ref{sec:nosquare}.

Another consequence of Theorem \ref{th:main1} applies to linear recurrences of sufficiently large order. Namely if $k \geq 6$, the existence of such a $C_\ell$ leads to a contradiction. The same holds already for $k = 5$, if we assume that the Pisot element $\alpha_1$ is not a unit in the ring of integers of $\mathbb{Q}(\alpha_1,\ldots,\alpha_k)$.
Thus, we obtain the following result.

\begin{theorem}\label{th:main2}
Let $(F_n)_{n\geq 0}$ be a linear recurrence sequence of integers which satisfies the properties of Theorem \ref{th:main1}.
Then there are only finitely many Diophantine triples $1<a<b<c$ with
$$
ab+1=F_x,\quad ac+1=F_y,\quad bc+1=F_z,
$$
with values in $\{F_n; n\geq 0\}$ if one of the following conditions are satisfied:
\begin{itemize}
\item[(i)] $k\ge 5$ and $\alpha$ is not a unit,
\item[(ii)] $k\ge 6$.
\end{itemize}
\end{theorem}

This theorem is proved in Sections \ref{sec:56} and \ref{sec:7}.

\section{Some useful lemmas}
\label{sec:2}

Assume we have infinitely many solutions $(x,y,z) \in \mathbb{N}^{3}$ to \eqref{eq:1} with $1<a<b<c$. Obviously, we have $x < y < z$.
First, one notices that not only for $z$, but for all three components, we necessarily have arbitrarily ``large" solutions.

\begin{lemma}\label{le:2}
\label{lem:xlarge}
Let us assume, we have infinitely many solutions $(x,y,z) \in \mathbb{N}^{3}$ to \eqref{eq:1}. Then for each $N$, there are still infinitely many solutions $(x,y,z) \in \mathbb{N}^{3}$ with $x > N$.
\end{lemma}
\begin{proof}
It is obvious that we must have arbitrarily large solutions for $y$ and for $z$, since otherwise, $a, b, c$ would all be bounded as well, which is an immediate contradiction to our assumption.

If we had infinitely many solutions $(x,y,z)$ with $x < N$, then there is at least one fixed $x$ which forms a solution with infinitely many pairs $(y,z)$. Since $F_x = ab + 1$, we have a bound on these two variables as well and can use the same pigeon hole argument again to find fixed $a$ and $b$, forming a Diophantine triple with infinitely many $c \in \mathbb{N}$.

Using these fixed $a, b$, we obtain from the other two equations in \eqref{eq:1}, that $b F_y - a F_z = b - a$ and therefore, the expressions $b f_1 \alpha_1^y$ and $y f_1 \alpha_1^z$ (having the largest growth rate) must be equal. So
$$
\alpha_1^{z-y} = \frac{b}{a},
$$
which is a constant. Hence, $z-y$ must be some constant $\oc > 0$ as well and we can write $z = y + \oc$ for our infinitely many solutions in $y$ and $z$.

Using the power sum representations in $b F_y - a F_{y+\oc} = b - a$, we get
\begin{equation}\label{eq:xgrow}
b (f_1 \alpha_1^y + \cdots + f_k \alpha_k^y) - a (f_1 \alpha_1^{y+\oc} + \cdots + f_k \alpha_k^{y+\oc}) = b-a.
\end{equation}
So the terms with the largest growth rate, which are $b f_1 \alpha_1^y$ and $a f_1 \alpha_1^{y+\oc}$ must be equal and this gives us $b = a \alpha_1^c$. Inserting this into \eqref{eq:xgrow} and cancelling on both sides gives us
$$
\alpha_1^{\oc} (f_2 \alpha_2^y + \cdots + f_k \alpha_k^y) - (f_2 \alpha_2^{y+\oc} + \cdots + f_k \alpha_k^{y+\oc}) = \alpha_1^{\oc}-1.
$$
Now for $y \to \infty$, the left hand side converges to $0$. The right hand side is a constant larger than $0$. So this equality can not be true when $y$ is large enough. This contradiction completes the proof.
\end{proof}

Next, we prove the following result, which generalizes Proposition 1 in \cite{FuLuSz1}. Observe that the upper bound depends now on $k$.

\begin{lemma}
\label{lem:1}
Let $y<z$ be sufficiently large. Then there is a constant $C_1$ such that
\begin{equation}
\label{eq:gcd}
\gcd(F_y-1,F_z-1)< C_1 \alpha_1^{\frac{k}{k+1}z}.
\end{equation}
\end{lemma}

\begin{proof}
Denote $g:=\gcd(F_y-1,F_z-1)$. Furthermore, let us assume that $y$ is large enough such that $|f_2\alpha_2^y + \dots + f_k \alpha_k^y| < 1/2$.  Let $\kappa$ be a constant to be determined later. If $y\le \kappa z$, then
\begin{equation}
\label{eq:ysmall}
g\le F_y-1<f_1 \alpha_1^y\le f_1 \alpha_1^{\kappa z}.
\end{equation}
Now let us assume that $y > \kappa z$. We denote $\lambda := z-y < (1 - \kappa) z$. Note that
$$
g\mid (F_z-1)-\alpha_1^{\lambda} (F_y-1)\qquad {\text{\rm in}}\qquad {\mathbb Q}(\alpha_1).
$$
Thus, we can write
$$
g\pi=(F_z-1)-\alpha_1^{\lambda} (F_y-1),
$$
where $\pi$ is some algebraic integer in ${\mathbb Q}(\alpha_1)$. Note that the right-hand side above is not zero, for if it were, we would get $\alpha_1^{\lambda}=(F_z-1)/(F_y-1)\in {\mathbb Q}$, which is false for $\lambda>0$. We compute norms from ${\mathbb  Q}(\alpha_1)$ to ${\mathbb Q}$. Observe that
\begin{equation*}\begin{split}
 \left|(F_z-1)-\right.&\left.\alpha_1^{\lambda}(F_y-1)\right|\\ & =  \left|(f_1\alpha_1^z+\cdots + f_k \alpha_k^z - 1)-\alpha_1^{\lambda}(f_1\alpha_1^{y}+\cdots + f_k \alpha_k^y - 1)\right|\\
& =  \left|\alpha_1^{\lambda}(1-f_2 \alpha_2^y - \cdots - f_k \alpha_k^y)-(1-f_2\alpha_2^z - \cdots -  f_k \alpha_k^z)\right| \\
& \le   \frac{3}{2} \alpha_1^{\lambda}-\frac{1}{2}<\frac{3}{2} \alpha_1^{\lambda}<\frac{3}{2} \alpha_1^{(1-\kappa)z}.
\end{split}\end{equation*}
Further, let $\sigma_i$ be any Galois automorphism that maps $\alpha_1$ to $\alpha_i$. Then for $i\ge 2$, we have
\begin{equation*}\begin{split}
\left|\sigma_i\left((F_z-1)-\alpha_1^{\lambda} (F_y-1)\right)\right| & =  \left|(F_z-1)-\alpha_i^{\lambda} (F_y-1)\right|\\
& <  F_z-1+F_y-1<\alpha_1^{z-1}+\alpha_1^{y-1}-2\\
& <  \left(1+\alpha_1^{-1}\right)\alpha_1^{z-1}\le C_2 \alpha^z,
\end{split}\end{equation*}
with $C_2$ being a suitable constant (e.g. $C_2 = \left(1+\alpha_1^{-1}\right)$).

Altogether, we obtain
\begin{equation*}\begin{split}
g^k & \le  g^k |N_{\K/{\mathbb Q}}(d\pi)|\\
& \le  \left| N_{\K/{\mathbb  Q}}((F_z-1)-\alpha_1^{\lambda} (F_y-1))\right|\\
& =  \left|\prod_{i=1}^k \sigma_i\left((F_z-1)-\alpha_1^{\lambda} (F_y-1)\right)\right|\\
& <  \frac{3}{2} \alpha_1^{(1-\kappa)z} C_2 (\alpha_1^{z})^{k-1}= C_3\alpha_1^{(k-\kappa)z}.
\end{split}\end{equation*}
Hence,
\begin{equation}
\label{eq:ylarge}
g \le C_4 \alpha_1^{(1-\kappa/k)x}.
\end{equation}
In order to balance \eqref{eq:ysmall} and \eqref{eq:ylarge}, we choose $\kappa$ such that $\kappa=1-\kappa/k$, giving $\kappa=k/(k+1)$ and
$$
g \le \max(f_1, C_4) \alpha_1^{\frac{k}{k+1}z} = C_1 \alpha_1^{\frac{k}{k+1}z},
$$
which proves the lemma.
\end{proof}

\section{Applying the Subspace theorem}
\label{sec:4}
The arguments in this section follow the arguments from \cite{FuLuSz1} and \cite{FuLuSz2}. We show that if there are infinitely many solutions to \eqref{eq:1}, then all of them can be parametrized by finitely many expressions as given in \eqref{eq:c} for $c$ below.

From now on, we may assume without loss of generality that $|\alpha_1| > |\alpha_2| \geq \cdots \geq |\alpha_k|$.

We assume that there are infinitely many solutions to \eqref{eq:1}. Then, for each integer solution $(a,b,c)$, we have
$$
a = \sqrt{\frac{(F_x - 1)(F_y - 1)}{F_z - 1}}, \, b = \sqrt{\frac{(F_x - 1)(F_z - 1)}{F_y - 1}}, \, c = \sqrt{\frac{(F_y - 1)(F_z - 1)}{F_x - 1}}.
$$
Our first aim is to prove, that the growth-rates of these infinitely many $x$, $y$ and $z$ have to be the same, except for a multiplicative constant. Let us recall that we trivially have $x < y < z$ and that, by Lemma \ref{le:2}, the solutions of $x$ need to diverge to infinity as well. We now want to prove that there exist constants $C_5, C_6 > 0$ such that $C_5z < y < C_6x$ for infinitely many triples $(x,y,z)$.
Let us choose $x, y, z$ large enough, such that $|f_2\alpha_2^x + \cdots + f_k \alpha_k^x| < 1$ and furthermore
$$
z > \frac{\log \left( \frac{2}{f_1(\alpha_1-1)}\right)}{\log \alpha_1} + 1,
$$
which implies $f_1 \alpha_1^{z-1}(\alpha_1 - 1) > 2$ and
$$
y > \frac{\log \left( \frac{2}{f_1^2(\alpha_1-1)}\right)}{\log \alpha_1},
$$
which implies $f_1^2 \alpha_1^{x + y}(\alpha_1 - 1) \geq f_1^2 \alpha_1^{y + 1}(\alpha_1 - 1) > 2$.
From the above, we obtain
$$
f_1^2 \alpha_1^{x+y+1} > f_1^2 \alpha_1^{x+y} + 2 \geq F_x F_y > (F_x - 1)(F_y - 1) \geq F_z - 1 \geq f_1 \alpha_1^{z} - 2 > f_1 \alpha_1^{z-1},
$$
which gives
$$
f_1\alpha_1^{x+y-1} \geq \alpha_1^{z - 1}.
$$
Choosing $C_7$ large enough, such that $\alpha_1^{C_7} > f_1$, we get $\alpha_1^{x+y+C_7-1} > \alpha_1^{z-1}$ and thus $x + y > z - C_7$.
For $z$ large enough, this gives
$$
2y > x + y \geq z - C_7 > \frac{z}{2},
$$
and thus $y > \frac{z}{4}$, which is the first inequality that we need.

In order to get a similar correspondence also between $x$ and $z$, we denote $d_1 := \gcd(F_y - 1, F_z - 1)$ and $d_2 := \gcd(F_x - 1, F_z - 1)$, such that $F_z - 1 \mid d_1 d_2$. Then we use Lemma \ref{lem:1} to obtain
$$
f_1 \alpha_1^{x-1} > F_x > F_x-1  \geq d_2 \geq \frac{F_z - 1}{d_1} \geq \frac{f_1 \alpha_1^{z} - 2}{C_8 \alpha_1^{\frac{kz}{k+1}}} \geq \frac{f_1}{C_8} \alpha_1^{\frac{z}{k+1}-1} > f_1 \alpha_1^{\frac{z}{k+1}-C_8}
$$
for $z$ large enough and hence
$$
x-1 > \frac{z}{k+1} - C_8
$$
which implies $x > C_9 z$ for a suitable new constant $C_9$ (depending only on $k$) and $x, z$ being sufficiently large.

Next, we do a Taylor series expansion for $c$ which was given by
\begin{equation}\label{eq:c-exp}
c = \sqrt{\frac{(F_y - 1)(F_z - 1)}{F_x - 1}}.
\end{equation}
Using the power sum representations of $F_x, F_y, F_z$, we get
\begin{equation*}\begin{split}
c = &\sqrt{f_1} \alpha_1^{(-x+y+z)/2} \\
& \cdot \left( 1 + (-1/f_1) \alpha_1^{-x} + (f_2/f_1) \alpha_2^x \alpha_1^{-x} + \dots + (f_k/f_1) \alpha_k^x \alpha_1^{-x}\right)^{-1/2} \\
& \cdot \left( 1 + (-1/f_1) \alpha_1^{-y} + (f_2/f_1) \alpha_2^y \alpha_1^{-y} + \dots + (f_k/f_1) \alpha_k^y \alpha_1^{-y} \right)^{1/2} \\
& \cdot \left( 1 + (-1/f_1) \alpha_1^{-z} + (f_2/f_1) \alpha_2^z \alpha_1^{-z} + \dots + (f_k/f_1) \alpha_k^z \alpha_1^{-z}\right)^{1/2}.
\end{split}\end{equation*}
We then use the binomial expansion to obtain
\begin{equation*}\begin{split}
\left( 1 \right.&\left.+ (-1/f_1) \alpha_1^{-x} + (f_2/f_1) \alpha_2^x \alpha_1^{-x} + \dots + (f_k/f_1) \alpha_k^x \alpha_1^{-x} \right)^{1/2} \\
&= \sum_{j=0}^T \binom{1/2}{j} \left((-1/f_1) \alpha_1^{-x} + (f_2/f_1) \alpha_2^x \alpha_1^{-x} + \dots + (f_k/f_1) \alpha_k^x \alpha_1^{-x} \right)^j \\
&\hspace*{1cm}+ \mathcal{O}(\alpha_1^{-(T+1) x}),
\end{split}\end{equation*}
where $\mathcal{O}$ has the usual meaning, using estimates from \cite{FuTi} and where $T$ is some index, which we will specify later.  Let us write ${\bf x}:=(x,y,z)$. Since $x < z$ and $z < x/C_9$, the remainder term can also be written as $\mathcal{O}(\alpha_1^{-T \| {\bf x}\|/C_9})$, where $\|{\bf x}\|=\max\{x,y,z\}=z$. Doing the same for $y$ and $z$ likewise and multiplying those expressions gives
\begin{equation}\label{eq:c-approx}
c=\sqrt{f_1} \alpha_1^{(-x+y+z)/2} \left( 1 + \sum_{j=1}^{n-1} d_j M_j\right) + \mathcal{O}(\alpha_1^{-T \|{\bf x}\|/C_9}),
\end{equation}
where the integer $n$ depends only on $T$, $d_j$ are non-zero coefficients in $\K={\mathbb Q}(\alpha_1,\ldots,\alpha_k)$, and $M_j$ is a monomial of the form
$$
M_j=\prod_{i=1}^k \alpha_i^{L_{i,j}({\bf x})},
$$
in which $L_{i,j}({\bf x})$ are linear forms in ${\bf x}\in {\mathbb R}^3$ with integer coefficients which are all non-negative if $i=2,\ldots,k$ and negative if $i=1$. Set $J=\{1,\ldots,n-1\}$.
Note that each monomial $M_j$ is ``small", that is there exists a constant $\kappa > 0$ (which we can even choose independently of $k$), such that
\begin{equation}\label{eq:Mj}
|M_j| \leq e^{-\kappa x}\qquad {\text{\rm for ~all}}\qquad j\in J.
\end{equation}
This follows easily from the following fact: By the Pisot property of $F_n$, we can write $|\alpha_1| > 1 + \zeta$ for a suitable $\zeta > 0$ (a conjecture of Lehmer
asserts that $\zeta$ can be chosen to be an absolute constant). Using this notation and a suitable $\kappa$, we have
\begin{equation*}\begin{split}
|M_j| &= |\alpha_1|^{L_{1,j}({\bf x})} \cdot |\alpha_2|^{L_{2,j}({\bf x})} \cdots |\alpha_k|^{L_{k,j}({\bf x})} \\
&\leq (1 + \zeta)^{L_{1,j}({\bf x})} \cdot 1 \cdots 1 \\
&\leq (1 + \zeta)^{-x} \\
&\leq e^{-\kappa x}\qquad {\text{\rm for ~all}}\qquad j\in J.
\end{split}\end{equation*}

Our aim of this section is to apply a version of the Subspace Theorem given in \cite{Ev} to show that there is a finite expansion of $c$ involving terms as in \eqref{eq:c-approx}; the version we are going to use can also be found in Section 3 of \cite{Fu2}, whose notation - in particular the notion of heights - we follow.

We work with the field $\K={\mathbb Q}(\alpha_1,\ldots,\alpha_k)$ and let $S$ be the finite set of places (which are normalized so that the Product Formula holds, cf. \cite{Ev}), that are either infinite or in the set $\{ v \in M_\K: |\alpha_1|_v \neq 1 \vee \cdots \vee |\alpha_k|_v \neq 1)\}$.
According to whether $-x+y+z$ is even or odd, we set $\epsilon = 0$ or $\epsilon = 1$ respectively, such that $\alpha_1^{(-x+y+z-\epsilon)/2} \in \K$. By going to a still infinite subset of the solutions, we may assume that $\epsilon$ is always either $0$ or $1$.

Using the fixed integer $n$ (depending on $T$) from above, we now define $n+1$ linearly independent linear forms in indeterminants $(C, Y_0, \dots, Y_n)$.
For the standard infinite place $\infty$ on $\C$, we set
\begin{equation}
l_{0, \infty}(C, Y_0, \dots, Y_{n-1}) := C - \sqrt{f_1 \alpha_1^\epsilon} Y_0 - \sqrt{f_1 \alpha_1^\epsilon} \sum_{j=1}^{n-1} d_{j} Y_j,
\end{equation}
where $\epsilon \in \{0,1\}$ is as explained above, and
$$
l_{i, \infty}(C, Y_0, \dots, Y_{n-1}) := Y_{i-1} \quad \textrm{for } i \in \{1, \dots, n\}.
$$
For all other places $v$ in $S$, we define
$$
l_{0, v} := C, \qquad l_{i, v} := Y_{i-1} \quad \textrm{ for } i = 1, \dots, n.
$$
We will show, that there is some $\delta > 0$, such that the inequality
\begin{equation}\label{eq:lf-ineq}
\prod_{v\in S}\prod_{i=0}^{n}\frac{\vert l_{i,v}({\bf
y})\vert_{v}}{\vert {\bf y}\vert_{v}} <\left(\prod_{v\in S}\vert
\det(l_{0,v},\ldots,l_{n,v})\vert_{v}\right)\cdot \mathcal{H}({\bf
y})^{-(n+1)-\delta}\,
\end{equation}
is satisfied for all vectors
$$
{\bf y} = (c, \alpha_1^{(-x+y+z-\epsilon)/2}, \alpha_1^{(-x+y+z-\epsilon)/2} M_1, \dots, \alpha_1^{(-x+y+z-\epsilon)/2} M_{n-1}).
$$
The use of the correct $\epsilon \in \{0,1\}$ guarantees that these vectors are indeed in $\K^n$.

First notice, that the determinant in \eqref{eq:lf-ineq} equals $1$ for all places $v$. Thus, \eqref{eq:lf-ineq} reduces to
$$
\prod_{v\in S}\prod_{i=0}^{n}\frac{\vert l_{i,v}({\bf
y})\vert_{v}}{\vert {\bf y}\vert_{v}} < \mathcal{H}({\bf
y})^{-(n+1)-\delta},
$$
and the double product on the left-hand side can be split up into
$$
\vert c - \sqrt{f_1 \alpha_1^\epsilon} y_0 - \sqrt{f_1 \alpha_1^\epsilon} \sum_{j=1}^{n-1} d_{j} y_j \vert_\infty \cdot \prod_{\substack{v \in M_{\K, \infty}, \\ v \neq \infty}} \vert c \vert_v \cdot \prod_{v \in S \backslash M_{\K, \infty}} \vert c \vert_v \cdot \prod_{j=1}^{n-1} \prod_{v \in S} \vert y_j \vert_v.
$$
Now notice that the last double product equals $1$ due to the Product Formula and that
$$
\prod_{v \in S \backslash M_{\K, \infty}} \vert c \vert_v \leq 1,
$$
since $c \in \mathbb{Z}$.
An upper bound on the number of infinite places in $\K$ is $k!$ and hence,
\begin{equation*}\begin{split}
\prod_{\substack{v \in M_{\K, \infty}, \\ v \neq \infty}} \vert c \vert_v &< \left( \frac{(T_y - 1) (T_z - 1)}{T_x - 1} \right)^{k!} \\
&\leq (f_1 \alpha_1^y + \cdots + f_k \alpha_k^y - 1)^{k!} (f_1 \alpha_1^z + \cdots + f_k \alpha_k^z - 1)^{k!} \\
&\leq (f_1 \cdot \alpha_1^{\| {\bf x} \|} + 1/2)^{2 \cdot k!}
\end{split}\end{equation*}
for $y$ large enough such that $|f_2 \alpha_2^y + \cdots + f_k \alpha_k^y| < 1/2$.
And finally the first expression is just
$$
\Big| \sqrt{f_1 \alpha_1^\epsilon}  \alpha_1^{(-x+y+z-\epsilon)/2} \sum_{j \geq n} d_j M_j \Big|,
$$
which, by \eqref{eq:c-approx}, is smaller than some expression of the form $C_{10} \alpha_1^{-T \|{\bf x}\| / C_9}$.
Therefore, we have
$$
\prod_{v\in S}\prod_{i=0}^{n}\frac{\vert l_{i,v}({\bf
y})\vert_{v}}{\vert {\bf y}\vert_{v}} < C_{10} \alpha_1^{-\frac{T \|{\bf x}\|}{C_{10}}} \cdot (f_1 \alpha_1^{\| {\bf x} \|} + 1/2)^{2 \cdot k!}.
$$

Now we choose $T$ (and the corresponding $n$) large enough such that
\begin{equation*}\begin{split}
C_{10} \alpha_1^{-\frac{T\|{\bf x}\|}{C_9}} &< \alpha_1^{-\frac{T\|{\bf x}\|}{2C_9}}, \\
(f_1 \alpha_1^{\|{\bf x}\|} + 1/2)^{2 \cdot k!} &< \alpha_1^{\frac{T\|{\bf x}\|}{4C_9}}.
\end{split}\end{equation*}
Then we can write
\begin{equation}
\prod_{v\in S}\prod_{i=0}^{n}\frac{\vert l_{i,v}({\bf y})\vert_{v}}{\vert {\bf y}\vert_{v}} < \alpha_1^{\frac{-T \|{\bf x}\|}{4C_9}}.
\end{equation}

For the height of our vector ${\bf y}$, we have the estimate
\begin{equation*}\begin{split}
\mathcal{H}({\bf y}) &\leq C_{11} \cdot \mathcal{H}(c) \cdot \mathcal{H}(\alpha_1^\frac{-x+y+z-\epsilon}{2})^n \cdot \prod_{i=1}^{n-1} \mathcal{H}(M_j) \\
&\leq C_{11}(f_1 \alpha_1^{\|{\bf x}\|} + 1/2)^{k!} \prod_{i=1}^{n-1} \alpha_1^{C_{12} \|{\bf x}\|} \\
&\leq \alpha_1^{C_{13} \|{\bf x}\|},
\end{split}\end{equation*}
with suitable constants $C_{11}, C_{12}, C_{13}$. For the second estimate, we used that
$$
\mathcal{H}(M_j) \leq \mathcal{H}(\alpha_1)^{C_{\alpha_1}({\bf x})}\mathcal{H}(\alpha_2)^{C_{\alpha_2}({\bf x})} \cdots \mathcal{H}(\alpha_k)^{C_{\alpha_k}({\bf x})}
$$
and bounded it by the maximum of those expressions. Furthermore we have
$$
\mathcal{H}(\alpha_1^\frac{-x+y+z-\epsilon}{2})^n \leq \alpha_1^{n \|{\bf x}\|},
$$
which just changes our constant $C_{13}$.

Now finally, the estimate
$$
\alpha_1^{-\frac{T \|{\bf x}\|}{4C_9}} \leq \alpha_1^{-\delta C_{13}\|{\bf x}\|}
$$
is satisfied provided that we pick $\delta$ small enough.

So all the conditions for the Subspace theorem are met. Since we assumed that there are infinitely many solutions $(x,y,z)$ of \eqref{eq:lf-ineq}, we now can conclude that all of them lie in finitely many proper linear subspaces. Therefore, there must be at least one proper linear subspace, which contains infinitely many solutions and we see that there exists a finite set $J_c$  and (new) coefficients $e_j$ (for $j\in J$) in $\K$ such that we have
\begin{equation}\label{eq:c}
c=\alpha_1^{(-x+y+z-\epsilon)/2} \left(e_0+\sum_{j\in J_c} e_j M_j\right)
\end{equation}
with monomials $M_j$ as before.

Likewise, we can find finite expressions of this form for $a$ and $b$.

\section{Parametrization of the solutions}\label{sec:5}

We use the following parametrization lemma:

\begin{lemma}
Suppose, we have infinitely many solutions for \eqref{eq:1}. Then there exists a line in $\mathbb{R}^3$ given by
\begin{displaymath}
x(t) = r_1 t + s_1 \quad y(t) = r_2 t + s_2 \quad z(t) = r_3 t + s_3
\end{displaymath}
with rationals $r_1, r_2, r_3, s_1, s_2, s_3$, such that infinitely many of the solutions $(x,y,z)$ are of the form $(x(n), y(n), z(n))$ for some integer $n$.
\end{lemma}

\begin{proof}
Assume that \eqref{eq:1} has infinitely many solutions. We already deduced in Section \ref{sec:4} that $c$ can be written in the form
\begin{displaymath}
c=\alpha_1^{(-x+y+z-\epsilon)/2} \left(e_{c,0}+\sum_{j\in J_c} e_{c,j} M_{c,j}\right)
\end{displaymath}
with $J_c$ being a finite set, $e_{c,j}$ being coefficients in $\K$ for $j \in J_c \cup \{0\}$ and $M_{c,j} = \prod_{i=1}^k \alpha_i^{L_{c,i,j}({\bf x})}$ with ${\bf x}=(x,y,z)$.
In the same manner, we can write
\begin{displaymath}
b=\alpha_1^{(x-y+z-\epsilon)/2} \left(e_{b,0}+\sum_{j\in J_b} e_{b,j} M_{b,j}\right).
\end{displaymath}
Since $1 + bc = F_z = f_1 \alpha_1^z + \cdots + f_k \alpha_k^z$, we get
\begin{equation}\label{eq:s-uniteq}
f_1 \alpha_1^z + \cdots + f_k \alpha_k^z - \alpha_1^{z-\varepsilon}  \left(e_{b,0}+\sum_{j\in J_b} e_{b,j} M_{b,j}\right)\left(e_{c,0}+\sum_{j\in J_c} e_{c,j} M_{c,j}\right) = 1.
\end{equation}

We now pick $\beta_1, \dots, \beta_\ell$ as a basis for the multiplicative group generated by $\{\alpha_1, \dots, \alpha_k, -1\}$, which we will denote with $S$. We express each $\alpha_1, \dots, \alpha_k$ as a product of $\beta_1, \dots, \beta_\ell$ and insert them into \eqref{eq:s-uniteq}. We obtain a new equation of the form
\begin{equation}\label{eq:s-uniteq2}
\sum_{j \in  J} e_j \beta_1^{L_{1,j}({\bf x})} \cdots \beta_{\ell}^{L_{\ell,j}({\bf x})} = 0,
\end{equation}
where again $J$ is some finite set, $e_j$ are new coefficients in $\K$ and $L_{i,j}$ are linear forms in ${\bf x}$ with integer coefficients. Note that the sum on the left hand side is not zero, since it contains the summand $-1$.
This is an $S$-unit equation.

We may assume that infinitely many of the solutions ${\bf x}$ are non-degenerate solutions of (\ref{eq:s-uniteq2}) by replacing the equation by a new equation given by a suitable vanishing subsum if necessary.

We may assume, that $(L_{1,i}, \dots, L_{\ell,i}) \neq (L_{1,j}, \dots, L_{\ell,j})$ for any $i \neq j$, because otherwise we could just merge these two terms.

Therefore for $i \neq j$, the theorem on non-degenerate solutions to $S$-unit equations (see \cite{ESS}) yields that the set of $$\beta_1^{L_{1,i}({\bf x}) - L_{1,j}({\bf x})} \cdots \beta_{\ell}^{L_{\ell,i}({\bf x}) - L_{\ell,j}({\bf x})}$$ is contained in a finite set of numbers.
Now since $\beta_1, \dots, \beta_{\ell}$ are multiplicatively independent, the exponents $(L_{1,i} - L_{1,j})({\bf x}), \ldots, (L_{\ell,i} - L_{\ell,j})({\bf x})$ take the same value for infinitely many ${\bf x}$. Since we assumed, that these linear forms are not all identically zero, this implies, that there is some non-trivial linear form $L$ defined over $\Q$ and some $c\in\mathbb{Q}$ with $L({\bf x}) = c$ for infinitely many ${\bf x}$.
So there exist rationals $r_i, s_i, t_i$ for $i = 1, 2, 3$ such that we can parametrise
\begin{displaymath}
x = r_1 p + s_1 q + t_1, \quad y = r_2 p + s_2 q + t_2, \quad z = r_3 p + s_3 q + t_3
\end{displaymath}
with infinitely many pairs $(p,q) \in \mathbb{Z}^2$.

We can assume, that $r_i, s_i, t_i$ are all integers. If not, we define $\Delta$ as the least common multiple of the denominators of $r_i, s_i$ ($i= 1, 2,3$) and let $p_0, q_0$ be such that for infinitely many pairs $(p, q)$ we have $p \equiv p_0 \mod \Delta$ and $q \equiv q_0 \mod \Delta$. Then $p = p_0 + \Delta \lambda, q = q_0 + \Delta \mu$ and
\begin{equation*}\begin{split}
x & =  (r_1\Delta) \lambda+(s_1\Delta) \mu+(r_1p_0+s_1q_0+t_1)\\
y & =  (r_2\Delta) \lambda+(s_2\Delta) \mu+(r_2p_0+s_2q_0+t_2)\\
z & =  (r_3\Delta) \lambda+(s_3\Delta) \mu+(r_3p_0+s_3q_0+t_3).
\end{split}\end{equation*}
Since $r_i \Delta$, $s_i \Delta$ and $x, y, z$ are all integers, $r_i p_0 + s_i q_0 + t_i$ are integers as well. Replacing $r_i$ by $r_i \Delta$, $s_i$ by $s_i \Delta$ and $t_i$ by $r_i p_0 + s_i q_0 + t_i$, we can indeed assume, that all coefficients $r_i, s_i, t_i$ in our parametrization are integers.

Using a similar argument as in the beginning of the proof, we get that our equation is of the form
\begin{displaymath}
\sum_{j \in J} e'_j \beta_1^{L'_{1,j}({\bf r})} \cdots \beta_{\ell}^{L'_{\ell,j}({\bf r})} = 0,
\end{displaymath}
where ${\bf r} := (\lambda, \mu)$, $J$ is a finite set of indices, $e_j'$ are new non-zero coefficients in $\K$ and $L'_{i,j}({\bf r})$ are linear forms in ${\bf r}$ with integer coefficients.
Again we may assume that we have $(L'_{1,i}({\bf r}), \dots, L'_{\ell,i}({\bf r})) \neq (L'_{1,j}({\bf r}), \dots, L'_{\ell,j}({\bf r}))$ for any $i \neq j$.

Applying the theorem of non-degenerate solutions to $S$-unit equations once more, we obtain a finite set of numbers $\Lambda$, such that for some $i \neq j$, we have
\begin{displaymath}
\beta_1^{(L'_{1,i} - L'_{1,j})({\bf r})} \cdots \beta_{\ell}^{(L'_{\ell,i} - L'_{\ell,j})({\bf r})} \in \Lambda.
\end{displaymath}
So every ${\bf r}$ lies on a finite collection of lines and since we had infinitely many ${\bf r}$, there must be some line, which contains infinitely many solutions, which proves our lemma.
\end{proof}

We apply this lemma and define $\Delta$ as the least common multiple of the denominators of $r_1, r_2, r_3$. Infinitely many of our $n$ will be in the same residue class modulo $\Delta$, which we shall call $r$.
Writing $n = m \Delta + r$, we get
\begin{displaymath}
(x,y,z) = ((r_1 \Delta) m + (r r_1 + s_1), (r_2 \Delta) m + (r r_2 + s_2), (r_3 \Delta) m + (r r_3 + s_3) ).
\end{displaymath}
Replacing $n$ by $m$, $r_i$ by $r_i \Delta$ and $s_i$ by $r r_i + s$, we can even assume, that $r_i, s_i$ are integers.
So we have

\begin{displaymath}
\frac{-x+y+z-\epsilon}{2} = \frac{(-r_1+r_2+r_3)m}{2} + \frac{-s_1+s_2+s_3 - \epsilon}{2}.
\end{displaymath}
This holds for infinitely many $m$, so we can choose a still infinite subset such that all of them are in the same residue class $\chi$ modulo $2$ and we can write $m = 2 \ell + \chi$ with fixed $\chi \in \{0,1\}$.
Thus, we have
\begin{displaymath}
\frac{-x+y+z-\epsilon}{2} = (-r_1 + r_2 + r_3) \ell + \eta,
\end{displaymath}
where $\eta \in \mathbb{Z}$ or $\eta \in \mathbb{Z} + 1/2$.

Using this representation, we can write \eqref{eq:c} as
$$
c(\ell) = \alpha_1^{(-r_1+r_2+r_3)\ell + \eta} \left(e_0+\sum_{j\in J_c} e_j M_j\right).
$$
for infinitely many $\ell$, where
$$
M_j=\prod_{i=1}^k \alpha_i^{L_{i,j}({\bf x})},
$$
as before and ${\bf x} = {\bf x}(\ell) = (x(2\ell + \chi), y(2\ell + \chi), z(2\ell + \chi))$.

So for infinitely many solutions $(x,y,z)$, we have a parametrization in $\ell$, such that $c$ is a power sum in this $\ell$ with its roots being products of $\alpha_1, \dots, \alpha_k$.
This, together with \eqref{eq:c-exp} gives the functional identity
\begin{equation}\label{eq:xcyz}
(F_x - 1) c^2 = (F_y - 1)(F_z - 1),
\end{equation}
which proves the main theorem.

\section{Linear recurrences with nonsquare leading coefficient}
\label{sec:nosquare}

The aim of this section is to prove Theorem \ref{co:nonsquare}.

\begin{proof}
We prove this result by contradiction: Suppose we had infinitely many Diophantine triples in $\{F_n; n\geq 0\}$. Then we can apply Theorem \ref{th:main1} as in the Sections \ref{sec:2}, \ref{sec:4} and \ref{sec:5} and obtain
\begin{equation}\label{eq:c-par}
c(\ell) = \alpha_1^{(-r_1+r_2+r_3)\ell + \eta} \left(e_0+\sum_{j\in J_c} e_j M_j\right)
\end{equation}
for infinitely many $\ell$, where
$$
M_j=\prod_{i=1}^k \alpha_i^{L_{i,j}({\bf x})},
$$
as before and ${\bf x} = {\bf x}(\ell) = (x(2\ell + \chi), y(2\ell + \chi)), z(2\ell + \chi))$.

First we observe, that there are only finitely many solutions of \eqref{eq:c-par} with
$c(\ell) = 0$. That can be shown by using the fact, that a simple non-degenerate linear recurrence has only finite zero-multiplicity (see \cite{ESS} for an explicit bound).
We will apply this statement here for the linear recurrence in $\ell$; it only remains to check, that no quotient of two distinct roots of the form \[\alpha_1^{L_{1,i}({\bf x}(\ell))} \cdots \alpha_k^{L_{k,i}({\bf x}(\ell))}\] is a root of unity or, in other words, that
\begin{equation*}\label{eq:unit-eq}
(\alpha_1^{m_1} \alpha_2^{m_2} \cdots \alpha_k^{m_k})^n = 1
\end{equation*}
has no solutions in $n \in \mathbb{Z}/\{0\}$, $m_1 < 0$ and $m_i > 0$ for $i = 2, \dots, k$. But this follows at once from Mignotte's result \cite{MM}.

So, we have confirmed that $c(\ell) \neq 0$ for still infinitely many solutions.
We use \eqref{eq:c-exp} and write
\begin{equation}\label{eq:xcyz}
(F_x - 1) c^2 = (F_y - 1)(F_z - 1).
\end{equation}
Then we insert the finite expansion \eqref{eq:c-par} in $\ell$ for $c$ into \eqref{eq:xcyz}.
Furthermore, we use the Binet formula
\begin{equation}\label{eq:Binet}
F_x = f_1 \alpha_1^x + \cdots + f_k \alpha_k^x
\end{equation}
and write $F_x, F_y, F_z$ as power sums in $x$, $y$ and $z$ respectively. Furthermore, we use the finite expansion of $c$, that we obtained by the Subspace theorem. We get an equation of the form
\begin{equation*}
\begin{split}
(f_1\alpha_1^x &+ \cdots + f_k \alpha_k^x - 1)\cdot \\
&\cdot \alpha_1^{-x+y+z-\epsilon} (e_0^2 + 2e_0 e_1 \alpha_1^{-x} + 2e_0 e_2 \alpha_1^{-y} + 2e_0 e_3 \alpha_1^{-z} + e_1^2 \alpha_1^{-2x} + \cdots) \\
&= (f_1 \alpha_1^y + \cdots + f_k \alpha_k^y - 1)(f_1 \alpha_1^z + \cdots + f_k \alpha_k^z - 1),
\end{split}
\end{equation*}
Using the parametrization $(x,y,z) = (r_1 m + s_1, r_2 m + s_2, r_3 m + s_3)$ with $m = 2\ell$ or $m = 2\ell+1$ as above, we have expansions in $\ell$ on both sides of \eqref{eq:xcyz}. Since there must be infinitely many solutions in $\ell$, the largest terms on both sides have to grow at the same rate.

In order to find the largest terms, let us first note the following: If $e_0 = 0$ for infinitely many of our solutions, then the largest terms were
\begin{equation}\label{eq:e0_0}
f_1 \alpha_1^x \alpha_1^{-x+y+z-\epsilon} e_1^2 \alpha_1^{-2x} = f_1 \alpha_1^y f_1 \alpha_1^z,
\end{equation}
or some even smaller expression on the left-hand side, if $e_1 = 0$ as well. Note that there could be more than one term in the expansion of $c$ with the same growth rate, for example if $y$ and $z$ are just translates of $x$ and thus $\alpha_1^{-y} = \alpha_1^{-x-c} = C \alpha_1^{-x}$, but this would only change the coefficient $e_1$ which we do not know anyway.
From \eqref{eq:e0_0}, we get
$$
e_1^2 \alpha_1^{-2x+y+z-\epsilon} = f_1 \alpha_1^{y+z}.
$$
Dividing by $\alpha_1^{y+z}$ on both sides, we see that the left-hand side converges to $0$, when $x$ grows to infinity (which it does by Lemma \ref{le:2}), while the right-hand-side is the constant $f_1 \neq 0$. This is a contradiction.

So we must have that $e_0 \neq 0$ for infinitely many of our solutions. Then $e_0 \alpha_1^{(-x+y+z-\epsilon)/2}$ certainly is the largest term in the expansion of $c$ and we have
\begin{displaymath}
f_1 \alpha_1^x \alpha_1^{-x+y+z-\epsilon} e_0^2 = f_1 \alpha_1^y f_1 \alpha_1^z.
\end{displaymath}
for the largest terms, which implies that $e_0^2 = f_1 \alpha_1^{\epsilon}$.
But this is a contradiction, since we assumed that neither $f_1$ nor $f_1 \alpha_1$ is a square in $\K$. So, the theorem is proved.
\end{proof}

\section{Linear recurrences of large order}
\label{sec:56}

We now prove Theorem \ref{th:main2}.

\begin{proof}
We follow the same notation as in the proof of Theorem \ref{th:main1}. Supposing that we have infinitely many Diophantine triples with values in $\{F_n; n\geq 0\}$, we get the functional identity
$$
(F_x - 1) C_\ell^2 = (F_y - 1) (F_z - 1),
$$
where $x=r_1\ell+s_1,y=r_2\ell+s_2,z=r_3\ell+s_3$, $r_1,r_2,r_3$ positive integers with gcd$(r_1,r_2,r_3)=1$ and $s_1,s_2,s_3$ integers.

We first handle (i) in the theorem. Therefore assume that $\alpha$ is not a unit. Then, by Mignotte's result \cite{MM}, there is no multiplicative dependence between the roots and thus (e.g. by using Lemma 2.1 in \cite{cz}), it follows that if we put ${\bf X}=(X_1,\ldots,X_k)$ and
$$
P_i({\bf X})=\sum_{j=1}^k f_j \alpha_j^{s_i} X_j^{r_i}-1\in \K[X_1,\ldots,X_k]\quad {\text{\rm for}}\quad  i=1,2,3,
$$
then for each $h\in \{1,2,3\}$ putting $i,j$ for the two indices such that $\{h,i,j\}=\{1,2,3\}$, we have that
$$
\frac{P_i({\bf X})P_j({\bf X})}{P_h({\bf X})}=Q_h({\bf X})^2,
$$
for some $Q_h({\bf X})\in \K[X_1^{\pm 1},\ldots,X_k^{\pm 1}]$. For this we have to identify the exponential function $\ell\mapsto \alpha_1^\ell$ by $X_1$, $\ell\mapsto\alpha_2^\ell$ by $X_2$ and so forth. Actually, Theorem \ref{th:main1} shows that $Q_h({\bf X})\in\K[X_1^{\pm 1},X_2,\ldots,X_k]$. This imposes some conditions on the degrees:
\medskip

(P) {\it Parity:} $r_1+r_2+r_3\equiv 0\pmod 2$. This is clear from degree considerations since $2{\text{\rm deg}_{X_2}}(Q_h)={\text{\rm deg}_{X_2}}(P_i)+{\text{\rm deg}_{X_2}} (P_j)-{\text{\rm deg}_{X_2}} (P_h)=r_i+r_j-r_h$.

\medskip

(T) {\it Triangular inequality:} $r_1+r_2> r_3$. It is clear that $r_1+r_2\ge r_3$, otherwise $P_1({\bf X})P_2({\bf X})/P_3({\bf X})$ has negative degree as a polynomial in, say, $X_2$, so it cannot be a polynomial in $X_2$.
To see that the inequality must be in fact strict, assume that equality holds. Then $Q_3({\bf X})=q_3\in \K[X_1^{\pm 1}]$.
Hence,
$$
P_1({\bf X})P_2({\bf X})=q_3^2 P_3({\bf X}).
$$
In the left, we have the monomial $X_1^{r_1} X_2^{r_2}$ with non-zero coefficient $f_1f_2 \alpha_1^{s_1} \alpha_2^{s_2}$, whenever  $r_1<r_2$. However, such monomials do not appear in the right above. Thus,
we must have $r_1=r_2$, and since further we also have $r_3=r_1+r_2$ and $\gcd(r_1,r_2,$ $r_3)=1$, it follows that $(r_1,r_2,r_3)=(1,1,2)$. In this case, the coefficient of $X_1X_2$ in the left is
$$
f_1f_2(\alpha_1^{s_1}\alpha_2^{s_2}+\alpha_2^{s_1}\alpha_1^{s_2}),
$$
and this must be zero since $X_1X_2$ does not appear in $P_3({\bf X})$. This shows that $$(\alpha_1/\alpha_2)^{s_1-s_2}=-1,$$ so $s_1=s_2$. But then  $x=y$, which is not allowed.

We will make more observations later. The key ingredient is the following lemma.

\begin{lemma}\label{le:7}
In case \emph{(i)} there are specializations $(X_3,\ldots,X_k)=(x_3,\ldots,x_k)$, a vector of non-zero algebraic numbers such that
$$
\sum_{j=3}^k f_j \alpha_j^{s_i} x_j^{r_i}=1,\quad {\text{for}}\quad i=1,2,3.
$$
The same is true in case \emph{(ii)} where we can further impose a condition of the form
$$
x_3\cdots x_k=1
$$
or
$$
x_3\cdots x_k=-1.
$$
\end{lemma}

The lemma will be proved in Section \ref{sec:7}. We go on and finish the proof of the theorem under the lemma. We specialize (i) as indicated in  the lemma to get that we may assume that $k=2$, ${\bf X}=(X_1,X_2)$ and
\begin{equation*}
\begin{split}
P_1({\bf X}) & =  f_1\alpha_1^{s_1} X_1^{r_1}+f_2 \alpha_2^{s_1} X_2^{r_1},\\
P_2({\bf X}) & =  f_1\alpha_1^{s_2}X_1^{r_2}+f_2 \alpha_2^{s_2} X_2^{r_2},\\
P_3({\bf X}) & =  f_1 \alpha_1^{s_3} X_1^{r_3}+f_2\alpha_2^{s_3} X_2^{r_3}.
\end{split}
\end{equation*}
Since $P_1,~P_2,~P_3$ are homogeneous, so are $Q_1,Q_2,Q_3$, so we may dehomogenize them, that is, assume that $X_2=1$ and we get some equalities of Laurent-polynomials of one variable. We then write
$R({\bf X})=R(X_1)$ since it no longer depends on $X_2$ after the convention that $X_2=1$.
Let us look at
$$
Q_3(X_1)^2=\frac{P_1(X_1) P_2(X_1)}{P_3(X_1)}.
$$
All roots of $P_i(X_1)$ are of the form $(-f_2/f_1)^{1/r_i} (\alpha_2/\alpha_1)^{s_i/r_i} \zeta_{r_i}$, for $i=1,2,3$, where $\zeta_{r_i}$ runs through a complete system of roots of unity of order $r_i$.
Now every root of $P_3(X_1)$ is either a root of $P_1(X_1)$ or a root of $P_2(X_1)$. Clearly, the number of common roots between $P_3(X_1)$ and $P_j(X_1)$ for $j=1,2$ is
at most $\gcd(r_3,r_j)$ for each of $j=1,2$. Thus, if $r_3>r_2$, then
\begin{itemize}
\item[(i)]  the number of common roots of $P_3(X_1)$ with $P_1(X_1)$ is $\le r_3/2$ with equality if and only if $r_1=r_3/2$;
\item[(ii)] the number of common roots of $P_3(X_1)$ with $P_2(X_1)$ is $\le r_3/2$ with equality if and only $r_2=r_3/2$.
\end{itemize}
Thus, we get
$$
r_3\le r_3/2+r_3/2,
$$
and the inequality is strict unless $r_1=r_2=r_3/2$. But the inequality cannot be strict, so $r_1=r_2=r_3/2$, and this leads to $(r_1,r_2,r_3)=(1,1,2)$, a case which we already saw that it is impossible in the proof of condition
Triangular inequality (T).

Thus, it is not possible that $r_3>r_2$, therefore $r_3=r_2$. Then, since $\operatorname{gcd}(r_1, r_2,$ $r_3) = 1$, we have $r_1<r_3$, so it follows that not all roots of $P_3(X_1)$ are roots of $P_1(X_1)$
by degree considerations.  Hence, $P_3(X_1)$ and $P_2(X_1)$ have a common root. Thus,
$$
(-f_2/f_1)^{1/r_2} (\alpha_2/\alpha_1)^{s_2/r_2} \zeta_{r_3}=(-f_2/f_1)^{1/r_2} (\alpha_2/\alpha_1)^{s_3/r_2}.
$$
 Canceling $(-f_2/f_1)^{1/r_2}$ and raising everything to power $r_2=r_3$, we get \[(\alpha_2/\alpha_1)^{s_3-s_2}=1,\] so $s_3=s_2$, and since also $r_2=r_3$, we get $y=z$, which is not allowed.
\medskip

In case (ii) of the theorem, the identification with Laurent-polynomials (e.g. again via Lemma 2.1 in \cite{cz}) does not work in the above form. But when $\alpha$ is a unit, then we have the relation $\alpha_1 \cdots \alpha_k = \pm 1$, which allows applying a similar identification but instead of ${\bf X}=(X_1,\ldots,X_k)$ by using
\begin{equation}
\label{eq:Xcaseii}
{\bf X}=(X_1,\pm 1/(X_1X_3\cdots X_k), X_3,\ldots,X_k).
\end{equation}
We insist again that by Mignotte's result \cite{MM} there are no other multiplicative relations between $\alpha_1,\ldots,\alpha_k$. In particular, any $k-1$ of these numbers (e.g. $\alpha_1,\alpha_3,\ldots,\alpha_k$) are multiplicatively independent. Hence, we may identify $\ell\mapsto\alpha_1^\ell$ by $X_1$, $\ell\mapsto\alpha_3^\ell$ by $X_3$ and so forth, which implies that $\ell\mapsto \alpha_2^\ell$ must be identified with $\pm 1/(X_1X_3\cdots X_k)$.
By the lemma there is a specialization $(X_3,\ldots,X_k)=(x_3,\ldots,$ $x_k)$ such that upon this specialization,
\begin{equation*}
\begin{split}
P_1(X_1) & =  f_1\alpha_1^{s_1} X_1^{r_1}+f_2\alpha_2^{s_1} X_1^{-r_1},\\
P_2(X_1) & =  f_1\alpha_1^{s_2} X_1^{r_2}+f_2\alpha_2^{s_2} X_1^{-r_2},\\
P_3(X_1) & =  f_1\alpha_1^{s_3} X_1^{r_3} +f_2 \alpha_2^{s_3} X_1^{-r_3}.
\end{split}
\end{equation*}
Factoring out $X_1^{-r_i}$ for $i=1,2,3$ and putting them on the $Q_k(X_1)$ side, and putting
$$
P_i'(X_1)=X_1^{r_i} P_i(X_1)=f_1\alpha_1^{s_i} X_1^{2r_i}+f_2\alpha_2^{s_i},\qquad i=1,2,3,
$$
we end up with
$$
(X_1^{(r_i+r_j-r_k)/2} Q_h(X_1))^2=\frac{P_i'(X_1) P_j'(X_1)}{P_h'(X_1)}
$$
and the left-hand side has no non-zero poles while the right-hand side does not have either a zero or a pole at $X_1=0$, so \[Q_h'(X_1)=X_1^{(r_i+r_j-r_k)/2} Q_h(X_1)\] is a polynomial which is not zero in $X_1=0$.
Note that the exponent $(r_i+r_j-r_h)/2$ of $X_1$ above is an integer by the Parity condition (P).
Now the conclusion can be reached in the same way as before noting that the right-hand side is invariant under $X_1\mapsto -X_1$, therefore so is the left-hand side, showing that $Q_h'(-X_1)=\pm Q_h'(X_1)$. Thus, $Q_h'(X_1)$
either has only monomials of odd degree or only monomials of even degree  but since $Q_h'(0)\not =0$, it follows that all monomials $Q_h'(X_1)$ are of even degree. Thus, $Q_h'(X_1)$ is a polynomial in $X_1^2$ for $h=1,2,3$,
and so are $P_i'(X_1)$ for $i=1,2,3$. Thus, we can make the substitution $X_1^2\mapsto X_1$
and then get to the problem treated in case (i).
\end{proof}

\section{Proof of lemma \ref{le:7}}
\label{sec:7}

In case (i), we take $(X_{7},\ldots,X_k)=(x_{7},\ldots,x_k)$, where these last $k-6$ numbers are algebraic and non-zero but otherwise arbitrary (we can take them rational, like equal to $1$, for example). Of course,
if $k=5$, we do nothing at this stage. Now choose $X_6=x_6$ such that $x_6\ne 0$ and
$$
f_6 \alpha_6^{s_i} x_6^{r_i}+\sum_{j=7}^k f_j \alpha_j^{s_i} x_j^{r_i}\ne 1,\qquad i=1,2,3.
$$
Thus, $x_6$ is chosen to be an algebraic (rational) number outside a finite set. Denote
$$
d_i=f_6 \alpha_6^{s_i} x_6^{r_i}+\sum_{j=7}^k f_j \alpha_j^{s_i} x_j^{r_i}-1\ne 0,\qquad i=1,2,3.
$$
It is now enough to show that there exist $x_3,x_4,x_5$ which are all non-zero and algebraic solving the following system
$$
f_3\alpha_3^{s_i} x_3^{r_i}+f_4\alpha_4^{s_4} x_4^{r_i}+f_5\alpha_5^{s_i} x_5^{r_i}+d_i=0,\qquad i=1,2,3.
$$
In case $r=5$, we have of course $d_i=-1$ for $i=1,2,3$.
We homogenize it by writing $x_i=z_i/z$ for $i=3,4,5$, getting
$$
f_3\alpha_3^{s_i} z_3^{r_i}+f_4\alpha_4^{s_i} z_4^{r_i}+f_5\alpha_5^{s_i} z_5^{r_i}+d_i z^{r_i}=0,\quad i=1,2,3.
$$
The above is a system of $3$ homogeneous polynomial equations in $4$ unknowns. Thus, by B\'ezout's theorem, it either has infinitely many solutions
(that is, if the set of common zeros contains some positive dimensional variety) or if not it has exactly $r_1r_2r_3$ projective solutions counting multiplicities.
Clearly, no point is a multiple solution since the gradient of the $i$th equation above is
$$
r_i(f_3\alpha_3^{s_i} z_3^{r_i-1}, f_4\alpha_4^{s_i} z_4^{r_i-1}, f_5\alpha_5^{s_i} z_5^{r_i-1},d_i z^{r_i-1})
$$
for any $i=1,2,3$, and this cannot be zero for $r_i>0$, since that would imply that $z_3=z_4=z_5=z=0$, but this does not  lead to a point $[z_3,z_4,z_5,z]$  in the projective space ${\mathbb P}_3({\mathbb C})$.
If we can prove that there is a solution with  $z\ne 0$, then we are done since we then can take $x_i=z_i/z$. So, let us assume that all solutions have $z=0$. The argument will be to show
that the number of them is $<r_1r_2r_3$, hence the surface at infinity $z=0$ cannot catch all solutions. So, if $z=0$, we get
$$
f_3 \alpha_3^{s_i} z_3^{r_i}+f_4\alpha_4^{s_i} z_4^{r_i}+f_5\alpha_5^{s_i}z_5^{r_i}=0,\qquad  {\text{\rm for}}\qquad i=1,2,3.
$$
We distinguish two cases.

\medskip

\noindent{\bf Case 1.} There exists $i\in\{3,4,5\}$ such that $z_i=0$.

\medskip

Say $z_5=0$. Then $z_3z_4\ne 0$ (since one of them being zero will imply that the other is also zero). Dehomogenizing and putting $w:=z_3/z_4$, we get
$$
f_3\alpha_3^{s_i} w^{r_i}+f_4\alpha_4^{s_i}=0,\qquad {\text{\rm for}}\qquad i=1,2,3.
$$
The solutions of the $i$th equation above  are $w=(-f_4/f_3)^{1/r_i} (\alpha_4/\alpha_3)^{s_i/r_i} \zeta_{r_i}$ for $i=1,2,3$, where $\zeta_{r_i}$ runs through a complete system of roots of unity of order $r_i$.
Since $\gcd(r_1,r_2,r_3)=1$, there can be at most one common root to all three equations. In fact, if two of the $r_1,r_2,r_3$ are equal then there is no common root. Indeed, say
$r_1=r_2$. Then the relation
$$
(-f_4/f_3)^{1/r_1} (\alpha_4/\alpha_3)^{s_1/r_1}\zeta_{r_1}=(-f_4/f_3)^{1/r_1} (\alpha_4/\alpha)^{s_2/r_1}\zeta_{r_2}
$$
holds for some roots of unity $\zeta_{r_1}, \zeta_{r_2}$ of order $r_1$. Canceling $(-f_4/f_3)^{1/r_1}$ and raising the relation to the power $r_1$, we get $(\alpha_4/\alpha_3)^{s_2-s_1}=1$.
Using Mignotte's result from \cite{MM}, it follows from $\alpha_3^{s_1-s_2} \alpha_4^{s_2-s_1} =1$ that $s_1=s_2$, which together with $r_1=r_2$ implies $x=y$, a contradiction.

Thus, this case can account for at most $3$ zeros and in fact for no zero at all if two of $r_1,r_2,r_3$ are equal.

\medskip

\noindent{\bf Case 2.} $z_i\ne 0$ for $i=3,4,5$.

\medskip

Putting $w_3=z_3/z_5,~w_4=z_4/z_5$, we are searching for solutions to
$$
f_3\alpha_3^{s_i} w_3^{r_i}+f_4\alpha_4^{s_i} w_4^{r_i}+f_5\alpha_5^{s_i}=0,\qquad {\text{\rm for}}\qquad  i=1,2,3.
$$
For each $i$, the above is a curve of degree $r_i$. The polynomial
$$
aX^m+bY^m+c
$$
is irreducible as a polynomial in ${\mathbb C}[X,Y]$ for $abc\ne 0$. In fact, via the birational transformation $(X,Y)\mapsto ((a/c)^{1/m} X,(b/c)^{1/m}Y)$, the above curve becomes associated to
$$
X^m+Y^m+1,
$$
the affine version of $X^m+Y^m+Z^m$, the set of zeros of which is the Fermat curve which is known to be irreducible and of maximal genus $(m-1)(m-2)/2$.
Thus, picking up already the first two curves for $i=1,~2$ of degrees $r_1,~r_2$ we get two irreducible curves. Unless they coincide, they can have at most $r_1r_2$ common points.
In order for them to coincide, we will need, by scaling the coefficients of $w_3^{r_i}$ to be $1$ for $i=1,2$, that $r_1=r_2$ and
$$
\left(\frac{f_4}{f_3} \left(\frac{\alpha_4}{\alpha_3}\right)^{s_1}, \frac{f_5}{f_3} \left(\frac{\alpha_5}{\alpha_3}\right)^{s_1}\right)=\left(\frac{f_4}{f_3} \left(\frac{\alpha_4}{\alpha_3}\right)^{s_2}, \frac{f_5}{f_3} \left(\frac{\alpha_5}{\alpha_3}\right)^{s_2}\right).
$$
This leads to $(\alpha_4/\alpha_3)^{s_2-s_1}=1$, so $s_1=s_2$, and since $r_1=r_2$, we get $x=y$, a contradiction.

\medskip

Having explored Cases 1 and 2, we get that the total number of common zeros with $z=0$ of our equations is at most
$$
3+r_1r_2,
$$
where $3$ does not appear if $r_1,r_2,r_3$ are not all distinct. However, the total number of zeros is at least $r_1r_2r_3$ and $r_1r_2r_3>r_1r_2$, otherwise $r_3=1$ so $(r_1,r_2,r_3)=(1,1,1)$, which is impossible by the
Parity condition (P). Thus, $r_1<r_2<r_3$, and we must have
$$
3+r_1r_2\ge r_1r_2r_3,\quad {\text{\rm so}}\quad 3\ge r_1r_2(r_3-1),
$$
which is impossible because $r_1r_2\ge 2$ and $r_3-1\ge 3-1=2$, so the right-hand side above is larger than $3$.

\medskip

This finishes the lemma in case (i).

\medskip

The proof in case (ii) is similar. In this case, we take arbitrary $(X_8,\ldots,X_k)=(x_8,\ldots,x_k)$ which are not zero, and then $X_7=x_7\ne 0$ such that
$$
f_7 \alpha_7^{s_i} x_7^{r_i}+\sum_{j=8}^k f_j \alpha_j^{s_i} x_j^{r_i}\ne 1,\qquad {\text{\rm for}}\qquad i=1,2,3.
$$
Putting
$$
d_i=f_7 \alpha_7^{s_i} x_7^{r_i}+\sum_{j=8}^k f_j \alpha_j^{s_i} x_j^{r_i} -1\ne 0,\qquad {\text{\rm for}}\qquad i=1,2,3,
$$
we need to show that there are solutions $x_3,x_4,x_5,x_6$ to
$$
f_3 \alpha_3^{s_i} x_3^{r_i}+f_4\alpha_4^{s_i} x_4^{r_i}+f_5\alpha_5^{s_i} x_5^{r_i}+f_6\alpha_6^{s_i}x_6^{r_i}+d_i=0,\qquad i=1,2,3,
$$
together with the additional equation
$$
x_3x_4x_5x_6-\frac{1}{x_7\cdots x_k}=0.
$$
As before, we homogenize it by writing $x_i=z_i/z$ for $i=3,4,5,6$ getting the four equations
\begin{eqnarray*}
f_3 \alpha_3^{s_i} z_3^{r_i}+f_4\alpha_4^{s_i} z_4^{r_i}+f_5\alpha_5^{s_i} z_5^{r_i}+f_6\alpha_6^{s_i}z_6^{r_i}+d_iz^{r_i} & = & 0,\qquad i=1,2,3,\\
z_3z_4z_5z_6 & = & \frac{z^4}{x_7\cdots x_k}.
\end{eqnarray*}
Now B\'ezout's theorem tells us that this system has at least $4r_1r_2r_3$ solutions. Again they are all simple. If $z=0$, then because of the last equation some
other variable should be $0$, so we can take $z_6=0$. Then, we just get
$$
c_3\alpha_3^{s_i} z_3^{r_i}+c_4 \alpha_4^{s_i} z_4^{r_i}+f_5 \alpha_5^{s_i} z_5^{r_i}=0,\qquad {\text{\rm for}}\qquad i=1,2,3,
$$
a case already treated at (i) for which we proved that it has at most $3+r_1r_2$ admissible solutions. Since $4r_1r_2r_3>3+r_1r_2$ (in all cases except $r_1=r_2=r_3=1$ which is not admissible by the Parity condition (P)),
this case also follows. Note that because of the last equation we get that $z\ne 0,$ which implies that all of $x_3,~x_4,~x_5,~x_6$ are non-zero.

The proof of case (ii) with the additional assumption $x_3 x_4 \cdots x_k = -1$ is the same.

\medskip

Theorem \ref{th:main2} now follows.\hfill$\square$

\end{document}